\documentclass[12pt]{amsart}
\usepackage[english]{babel}
\usepackage{paralist,url,verbatim, anysize}
\usepackage{amscd}
\usepackage[all,cmtip]{xy} 
\usepackage{amsmath,amsthm,amssymb,amsfonts}
\usepackage[bookmarksopen=true]{hyperref}
\usepackage[usenames, dvipsnames]{color}
\usepackage{graphicx}
\usepackage{hyperref}
\usepackage{euler, amsfonts, amssymb, latexsym, epsfig,epic}
\usepackage{xcolor}

\makeatletter
\let\@fnsymbol\@arabic
\makeatother

\theoremstyle{plain}
\newtheorem{theorem}{\bf Theorem}[section]

\newtheorem{lemma}[theorem]{Lemma}
\newtheorem{proposition}[theorem]{Proposition}

\theoremstyle{definition}

\newtheorem{definition}[theorem]{Definition}

\newtheorem{remark}[theorem]{Remark}

\newtheorem*{theorem*}{\bf Theorem}

  \newcommand{\ini}{\operatorname{in} }

\newcommand{\NN}{\mathbb{N}}
\newcommand{\ZZ}{\mathbb{Z}}

\newcommand{\FF}{\mathbb{F}}

\newcommand{\Tor}{\operatorname{Tor} }

\newcommand{\mm}{\mathfrak{m}}

\newcommand{\reg}{\operatorname{reg} }

\definecolor{mypink}{RGB}{215, 5, 234}

 \newcommand{\Rees}{\mathrm{Rees}}  
 \newcommand{\grade}{\mathrm{grade}}

\begin{document}

\title{Castelnuovo-Mumford regularity and powers}

\author{Winfried Bruns}
\address{Universit\"at Osnabr\"uck, Institut f\"ur Mathematik, 49069 Osnabr\"uck, Germany}
\email{wbruns@uos.de}
\author{Aldo Conca}
\email{conca@dima.unige.it}
\address{Dipartimento di Matematica, Universit\'a di Genova, Italy}
\author{Matteo Varbaro} 
\email{varbaro@dima.unige.it}
\address{Dipartimento di Matematica, Universit\'a di Genova, Italy} 
 \thanks{AC and MV  were  partially supported by INdAM-GNSAGA} 
  \date{}
 
\dedicatory{Dedicated to David Eisenbud on the occasion of his seventy-fifth birthday.}

\maketitle

\begin{abstract} This note has two goals.  The first is  to give a short and self contained introduction to the Castelnuovo-Mumford regularity for standard graded rings $R=\bigoplus_{i\in \NN}  R_i$ over general base rings $R_0$.  The second is to present a simple and concise proof of a  classical result due to Cutkosky, Herzog and Trung  and,  independently,  to  Kodiyalam asserting that the regularity of powers $I^v$ of an homogeneous ideal $I$ of $R$ is eventually a linear function in $v$. Finally we show  how the flexibility  of the definition of the Castelnuovo-Mumford regularity over general base rings can be used to give a simple proof of a result proved by the authors in \cite{BCV}. 
\end{abstract}

\section{Castelnuovo-Mumford regularity over general base rings} 
\label{CM}
Castelnuovo-Mumford regularity was introduced in the early eighties of the twentieth century by Eisenbud and Goto in \cite{EG} and by Ooishi \cite{O} 
as an algebraic counterpart of  the notion of regularity  for coherent sheaves on projective spaces discussed by Mumford in \cite{M}. 

One of the most important features of Castelnuovo-Mumford regularity is that it can be equivalently defined in terms of  (and hence  it bounds)  the vanishing of local cohomology modules,   the vanishing of  Koszul homology  modules and 
 the vanishing of syzygies.  

This triple nature of Castelnuovo-Mumford regularity is usually stated for graded rings over base fields, but indeed it holds in general as we will show in this section.   

Let $R=\bigoplus_{i\in \NN}  R_i$ be a $\NN$-graded ring with $R_0$  commutative and Noetherian. We assume that  $R$ is  standard graded, i.e., it is generated as an $R_0$-algebra by finitely many elements $x_1,\dots, x_n$ of degree $1$.  
Let $S=R_0[X_1,\dots, X_n]$  with $\NN$-graded structure induced by the assignment $\deg X_i=1$. 
The $R_0$-algebra map $S\to R$  sending $X_i$ to $x_i$ induces an $S$-module structure on $R$ and hence on every $R$-module.

Let $M=\bigoplus_{i\in \ZZ}  M_i$ be  a finitely generated graded $R$-module. Given $a\in \ZZ$ we will denote by  $M(a)$ the module that it is obtained from $M$ by shifting the degrees by $a$, i.e. $M(a)_i=M_{i+a}$.  

The Castelnuovo-Mumford regularity of $M$ is defined in terms of local cohomology modules $H^i_{Q_R}(M)$ with support on 
$$Q_R=R_+=(x_1,\dots, x_n).$$ 
For general properties of local cohomology modules we refer the readers to 
\cite{BS, BH, E}.  
In our setting the module $H^i_{Q_R}(M)$ is $\ZZ$-graded and its homogeneous component $H^i_{Q_R}(M)_j$ of degree $j\in \ZZ$ vanishes for large $j$.  The Castelnuovo-Mumford regularity of $M$  or, simply, the regularity of $M$ is defined as 
$$\reg(M)=\max\{  i+j : H_{Q_R}^i(M)_j\neq 0\}.$$
We may as well consider $M$ as an $S$-module by means of the map $S\to R$  and local cohomology supported on
$$Q_S=(X_1,\dots, X_n).$$
Since $H_{Q_S}^i(M)=H_{Q_R}^i(M)$ the resulting regularity is the same.  

 Here we list some simple properties of regularity that we will freely use.

\begin{itemize} 
\item[(1)]  $\reg(M(-a))=\reg(M)+a$. 
\item[(2)] $\reg(S)=0$ because $H_{Q_S}^i(S)=0$ for $i\neq n$ and $H_{Q_S}^n(S)=(X_1\cdots X_n)^{-1}R_0[X^{-1}_1,\dots, X^{-1}_n]$.
\item[(3)] If  $0\to N\to M\to L\to 0$ 
is a short exact sequence of finitely generated graded $R$-modules with maps of degree $0$ then : 
$$\begin{array}{rl}
\reg(N)& \leq \max\{ \reg(M), \reg(L)+1 \}, \\
\reg(M)& \leq \max\{ \reg(L), \reg(N)  \}, \\
\reg(L)& \leq \max\{ \reg(M), \reg(N)-1\}. 
\end{array} 
$$
\end{itemize} 

A minimal set of generators of $M$ is, by definition,  a set of generators that is minimal with respect to inclusion. The number of elements in a  minimal set of generators is not uniquely determined, but the set of the degrees of the elements in a minimal set of homogeneous generators of $M$ is uniquely determined because it coincides with the set of  $i\in \ZZ$  such that $[M/Q_RM]_i\neq 0$.  So we have a well defined notion of largest degree of a minimal generator of $M$ that we denote by $t_0(M)$, that is,
$$t_0(M)=\max\{ i\in \ZZ : [ M/Q_RM ]_i\neq 0\}$$
if $M\neq 0$. 
We use $t_0$ because  $M/Q_RM\simeq \Tor^R_0(M, R_0)=\Tor^S_0(M, R_0)$.  

The following result establishes the crucial link between the regularity and the degree of generators of a module. It appears in \cite[Thm.2]{O},  where it is attributed to Mumford,  and it appears also in \cite[Thm.16.3.1] {BS}. 

\begin{lemma}
\label{mainLemma}
$t_0(M)\leq \reg(M)$. 
\end{lemma} 

\begin{proof} Let $v=t_0(M)$. Then the $R_0$-module   $[M/Q_SM]_v$ is non-zero. Therefore there is a prime ideal $P$ of $R_0$ such that $[M/Q_SM]_v$  localized at $P$ is non-zero. In other words,  the localization $M'$ of $M$ at the multiplicative set $R_0\setminus P$  is a graded module over $(R_0)_P[X_1,\dots, X_n]$ with $t_0(M')=t_0(M)$. Since $\reg(M')\leq \reg(M)$ we may assume right away that $R_0$ is local with maximal ideal, say, $\mm$. Similarly we may also assume that the residue field of $R_0$ is infinite. 
If $M=H_{Q_S}^0(M)$, the assertion is obvious.  If $M\neq H_{Q_S}^0(M)$ then set $M'=M/H_{Q_S}^0(M)$. Clearly $t_0(H_{Q_S}^0(M))\leq \reg(M)$ and $\reg(M')\leq \reg(M)$. Since $t_0(M)=\max\{t_0(M'),t_0(H_{Q_S}^0(M))\}$ it is enough to prove the statement for $M'$. That is to say, we may assume that $\grade(Q_S,M)>0$. Because the residue field of $R_0$ is infinite, there exists $L\in S_1\setminus \mm S_1$ such that $L$ is a non-zero-divisor on $M$. By a change of coordinates we may assume that $L=X_n$. 
The short exact sequence 
$$0\to M(-1)\to M\to \overline{M}=M/(X_n)M\to 0$$
implies that $\reg(\overline{M})\leq \reg(M)$ (it is actually equal but we do not need it). As $\overline{M}$ is a finitely generated graded module over $R_0[X_1,\dots,X_{n-1}]$, we may assume, by induction on the number of variables, that it is generated in degree $\leq \reg(M)$. But then it follows easily that also $M$ is generated in degree $\leq \reg(M)$.
\end{proof} 
 
 Next  we consider the (graded) Koszul homology $H(Q_R,M)=H(Q_S,M)$  and set:
$$\reg_1(M)=\max\{ j-i :  H_i(Q_R,M)_j\neq 0  \}.$$
In this case, since $H_0(Q_R,M)\cong M/Q_RM$, the assertion  
$$t_0(M)\leq \reg_1(M)$$
 is obvious.
Now, let 
$$\FF: \dots  \to F_c \to F_{c-1} \to \dots \to F_0 \to 0$$
be a  graded  $S$-free resolution of $M$, i.e.,  each $F_i$ is a  graded and $S$-free of finite rank, the maps have degree $0$ and $H_i(F)=0$ for all $i$ with the exception of $H_0(\FF)\simeq M$. 
We say that $\FF$ is minimal if  a basis of $F_0$  maps to a minimal set of homogeneous generators of $M$,  a basis of $F_1$  maps to  a minimal set of homogeneous generators of the kernel of $F_0\to M$ and for $i\geq 2$ a basis of   $F_i$ maps to  a minimal set of homogeneous generators of the kernel of $F_{i-1}\to F_{i-2}$.  

If $R_0$ is a field then a (finite) minimal $S$-free resolution always exists and it is unique up to an isomorphism of complexes. 
For general $R_0$, it is still true that  every module has a minimal free graded resolution but it is, in general, not  finite and furthermore  it is not  unique up to an isomorphism of complexes. 

Given a minimal graded $S$-free resolution $\FF$ of $M$ we set: 
$$\reg_2(\FF)=\max\{ t_0(F_i)-i \  :  \  i=0,\dots, n-\grade(Q_S,M)  \}$$
and 
$$\reg_3(\FF)=\max\{ t_0(F_i)-i : i\in \NN\}.$$
Obviously we have $t_0(M)\leq \reg_2(\FF)\leq \reg_3(\FF)$. 
We are ready to establish the following fundamental result: 

\begin{theorem} \label{t:main}
With the notation above and for every minimal $S$-free resolution $\FF$ of $M$,  we have: 
$$\reg(M)=\reg_1(M)=\reg_2(\FF)=\reg_3(\FF).$$
\end{theorem} 

\begin{proof}  
Set $Q=Q_S$ and $g=\grade(Q,M)=\min\{ i:  H^i_{Q}(M)\neq 0\}$. 

\vspace{2mm}

We first prove that $\reg(M)\leq \reg_1(M)$. 
 We prove the statement by decreasing induction on $g$.  Suppose  $g=n$. The induced map $H_Q^n(F_0)\to H_Q^n(M)$ is surjective. Hence   we have   
 $$\reg(M)\leq \reg(F_0)=t_0(F_0)=t_0(M)=\max\{ j : H_0(Q,M)_j\neq 0\}=\reg_1(M).$$
Now assume that $g<n$ and consider 
$$0\to M_1\to F_0\to  M\to 0.$$
 We have $\grade(Q,M_1)=g+1$ and 
 $$\reg(M)\leq \max\{ \reg(F_0), \reg(M_1)-1\}.$$  By induction 
 $\reg(M_1)\leq \reg_1(M_1)$. Since $H_i(Q,M_1)=H_{i+1}(Q,M)$ for $i>0$ and 
 $$0\to H_1(Q,M)\to H_0(Q,M_1)\to H_0(Q,F_0)\to H_0(Q,M)\to 0$$ is an exact sequence, we have 
 $$\reg_1(M_1)=\max\{ j-i : H_i(Q,M_1)_j\neq 0\}=\max\{a,b\}$$
 with $a= \max\{ j : H_0(Q,M_1)_j\neq 0\}$ and $b=\max\{ j-i : H_{i+1}(Q,M)_j\neq 0 \mbox{ and } i>0\}$. So $b\leq \reg_1(M)+1$ and, since $a\leq \max\{ t_0(F_0),  \max\{ j : H_1(Q,M)_j\neq 0\} \}$, we have that 
 $a\leq \reg_1(M)+1$ as well. Hence 
 $$\reg_1(M_1)\leq \reg_1(M)+1$$ 
and it follows that $\reg(M)\leq \reg_1(M).$

\vspace{2mm}

Secondly we prove that $\reg_1(M)\leq \reg_2(\FF)$. Since 
$$H_i(Q,M)=\Tor^S_i(M,R_0)=H_i(\FF  \otimes R_0)$$
we have that $H_i(Q,M)$ is a subquotient of $F_i\otimes R_0$ and hence 
$$\max\{ j : H_i(Q,M)_j\neq 0\}\leq t_0(F_i).$$ 
Furthermore,   $H_i(Q,M)=0$ if $i> n-g$. Therefore  $\reg_1(M)\leq \reg_2(\FF)$. 

\vspace{2mm}

That $\reg_2(\FF)\leq \reg_3(\FF)$ is obvious by definition, so ti remains to prove that  $\reg_3(\FF)\leq \reg(M)$.  Set $M_0=M$ and consider the exact sequence
$$0\to M_{i+1} \to F_i\to M_i \to 0. $$
By the minimality of $\FF$ we have $t_0(F_i)=t_0(M_i)\leq \reg(M_i)$. Hence 
$$\reg(M_{i+1})\leq \max\{ t_0(F_i), \reg(M_i)+1\}=\reg(M_i)+1$$
for all $i\geq 0$. It follows that 
$$t_0(F_i)=t_0(M_i)\leq \reg(M_i)\leq \reg(M)+i$$ 
for every $i$, that is, 
$$t_0(F_i)-i \leq \reg(M),$$
in other words,
$$\reg_3(\FF)\leq  \reg(M).$$
\end{proof} 

\begin{remark}\label{r:change}
Let $T\to R_0$ be any surjective homomorpism of unitary rings. It extends uniquely to 
$S'=T[X_1,\dots, X_n]\to S=R_0[X_1,\dots, X_n]$. Therefore a finitely generated  graded $R$-module $M$ can be regarded as a finitely generated  graded $S'$-module. Hence the regularity of  $M$ can be computed also using a graded minimal free resolution as $S'$-module. 
\end{remark}

\section{Bigraded Castelnuovo-Mumford regularity} 
\label{BigCM}
Assume now  $R=\bigoplus_{(i,j) \in \NN^2}  R_{(i,j)}$ is $\NN^2$-graded  with $R_{(0,0)}$  commutative and Noetherian and that  $R$ is  generated as an $R_{(0,0)}$-algebra by  elements $x_1,\dots, x_n, y_1,\dots, y_m$  with the $x_i$ homogeneous of degree $(1,0)$ and the $y_j$ homogeneous of degree $(0,1)$

We will denote by $R^{(*,0)}$ the subalgebra $\bigoplus_{i} R_{(i,0)}$ of $R$ and by $Q_{(1,0)}$ the ideal of $R^{(*,0)}$ generated by  $R_{(1,0)}$ i.e., by $x_1,\dots, x_n$. 
Similarly $R^{(0,*)}$ is the subalgebra $\bigoplus_{j} R_{(0,j)}$   
of $R$ and $Q_{(0,1)}$ the ideal of $R^{(0,*)}$ generated by  $R_{(0,1)}$ i.e., by $y_1,\dots, y_m$.  
We have (at least) three ways of getting an $\NN$-graded structure out of the  $\NN^2$-graded structure: 

\begin{itemize} 
\item[(1)] $(1,0)$-graded structure:  the homogeneous component of degree $i\in \NN$ is  given by $R^{(i,*)}=\bigoplus_{j} R_{(i,j)}$.  The degree $0$ part is $R^{(0,*)}$ and the ideal of the homogeneous elements of positive degree is  $Q_{(1,0)}R=(x_1,\dots, x_n)$.  
\item[(2)] $(0,1)$-graded structure:  the homogeneous component of degree $j\in \NN$ is  given by $R^{(*,j)}=\bigoplus_{i} R_{(i,j)}$.  The degree $0$ part is $R^{(*,0)}$ and the ideal of the homogeneous elements of positive degree is  $Q_{(0,1)}R=(y_1,\dots, y_m)$.  
\item[(3)] total degree: the homogeneous component of degree $u\in \NN$ is  $\bigoplus_{i+j=u} R_{(i,j)}$. The degree $0$ part is $R_{(0,0)} $ and the ideal of the homogeneous elements of positive degree is   $(x_1,\dots, x_n, y_1,\dots, y_m)$.
\end{itemize} 
In the same way,  any $\ZZ^2$-graded $R$-module  $M=\bigoplus M_{(i,j)}$  can be turned into a $\ZZ$-graded module by regrading it  with respect  to the $(1,0)$-grading or with respect the $(0,1)$-grading or with respect to the total degree.  

We may  hence consider the Castelnuovo-Mumford regularity of $M$ with respect to any of these different graded structures. To distinguish them we will denote by $\reg_{(1,0)} M$ the regularity of $M$ with respect to the $(1,0)$-graded structure and by $\reg_{(0,1)} M$ the regularity of $M$ with respect to the $(0,1)$-graded structure.

Given $i,j\in \ZZ$ we set $M^{(i,*)}=\bigoplus_v  M_{(i,v)}$ and $M^{(*,j)}=\bigoplus_v  M_{(v,j)}$. 
Clearly $M=\bigoplus_i  M^{(i,*)}$ as a $R^{(0,*)}$-graded module and $M=\bigoplus_j  M^{(*,j)}$ as an $R^{(*,0)}$-graded module. Also, it is simple to check that, if $M$ is a finitely generated $\ZZ^2$-graded module, then $M^{(i,*)}$ is a finitely generated $R^{(0,*)}$-graded module for all $i\in\ZZ$ and $M^{(*,j)}$ is a finitely generated $R^{(*,0)}$-graded module for all $j\in\ZZ$.

Let $S=R_{(0,0)}[X_1,\dots, X_n,Y_1,\dots,Y_m]$  with the $\NN^2$-graded structure induced by the assignment $\deg X_i=(1,0)$ and  $\deg Y_j=(0,1)$. We have: 

\begin{proposition}\label{p:bigraded} Let $M$ be a finitely generated $\ZZ^2$-graded $R$-module. Let $\FF$ be a bigraded $S$-free minimal resolution of $M$. Let $v_i$ be the largest integer $v$ such that $F_i$ has a minimal generator in degree $(v,*)$ and $w_i$ be the largest integer $w$ such that $F_i$ has a minimal generator in degree $(*,w)$ Then we have 
\begin{align*}
\max\{ \reg M^{(*,j)} : j\in \ZZ \}&= \reg_{(1,0)} M=  \max\{ v_i-i  : i=0,\dots, n\},\\
\max\{ \reg M^{(i,*)} : i\in \ZZ \}&=\reg_{(0,1)} M=   \max\{ w_i-i : i=0,\dots, m\},
\end{align*}
where $\reg M^{(*,j)}$ is the regularity as an $R^{(*,0)}$-graded module and $\reg M^{(i,*)}$ is the regularity as an $R^{(0,*)}$-graded module. 
\end{proposition} 

\begin{proof} Set $Q=Q_{(1,0)}$, i.e. $Q$ is  the ideal of $R^{(*,0)}$ generated by  $R_{(1,0)}$. The $(1,0)$-regularity of $M$ is defined by means of  the local cohomology  $H_{QR}^*(M)$. We may regard $M$ as an $R^{(*,0)}$-module, so that $H_{QR}^c(M)=H_{Q}^c(M)=\bigoplus_j H_{Q}^c(M^{(*,j)})$ for all $c$. This explains the first equality. For the second equality,  by Theorem \ref{t:main}  $\reg_{(1,0)} M$ can be computed from any graded minimal  free resolution of $M$ as an $R^{(0,1)}[X_1,\dots, X_n]$-module but we have observed in Remark \ref{r:change} that it can be as well computed from any minimal  free resolution of $M$ as an $S$-module. So a minimal bigraded resolution of $M$ as $S$-modules serves to compute both the $(1,0)$ and the $(0,1)$ regularity. 
\end{proof}

\section{A non-standard $\ZZ^2$-grading}
\label{nonstan}
For later applications we will consider in this section a polynomial ring 
$$A= A_0[Y_1,\dots, Y_g]$$
over a ring $A_0$ with a (non-standard)   $\ZZ^2$-graded structure given by 
$$\deg Y_j=(d_j,1)$$
where $d_1,\dots, d_g\in \NN$. 
 
For every  $\ZZ^2$-graded $A$-module $N=\bigoplus N_{(i,v)}$ and for every $v\in \ZZ$ we set 
$$\rho_N(v)=\sup\{ i\in \ZZ  : N_{(i,v)}\neq 0\}\in \ZZ\cup\{\pm \infty\}.$$
We will study the behaviour of  $\rho_N(v)$ as a function of $v$.  
We start with two general facts. 
\begin{lemma} 
\label{easyrho}
Given an chain of submodules $0=N_0\subset N_1 \subset N_2  \subset \dots \subset N_p=N$ of $\ZZ^2$-graded $A$ modules one has $\rho_N(v)=\max\{ \rho_{N_i/N_{i-1}}(v) : i=1,\dots, p\}$ for all $v$.  
\end{lemma} 
\begin{proof}  The function $\rho_N(v)$ behaves well on short exact sequences with maps of degree $0$. Then the statement follows by induction on $p$ using the short exact sequences associated to the chain of submodules.  
\end{proof} 

Let $F$ be a finitely generated  $\ZZ^2$-graded  free $A$-module with basis $e_1,\dots, e_p$ and let $<$  be a monomial order  on $F$. 
For every  $\ZZ^2$-graded $A$-submodule $U$ of $F$ we denote by $\ini_<(U)$ the $A_0$-submodule of $F$ generated by leading monomials (with coefficients!) of the non-zero elements in $U$.  Since $U$ is an $A$-submodule  of $F$, it turns out that $\ini_<(U)$ is an $A$-submodule of $F$ as well.  Furthermore for every monomial $aY^\alpha e_i$ in $\ini_<(U)$ there exists an element $u\in U$ such that $\ini_<(u)=aY^\alpha e_i$. 
One has: 

\begin{lemma}  
$\rho_{F/U}(v)=\rho_{F/\ini_<(U)}(v)$ for all $v$.
\end{lemma}
\begin{proof}  It is enough to prove that, given $(i,v)$, one has  $U_{(i,v)}=F_{(i,v)}$  if and only if $\ini_<(U_{(i,v)})=F_{(i,v)}$ 
The ``only if" implication is obvious. For the ``if" implication, we argue by contradiction. Suppose $\ini_<(U_{(i,v)})=F_{(i,v)}$ and $U_{(i,v)} \not =F_{(i,v)}$.  Let $Y^\alpha e_i$ be the smallest (with respect to the monomial order) monomial of degree $(i,v)$ which is not in $U_{(i,v)}$. Since $Y^\alpha e_i\in \ini_<(U_{(i,v)})$ 
there exists  $u\in U$ such that  $\ini_<(u)=Y^\alpha e_i $. We may assume that $u$ is homogeneous of degree $(i,v)$. If not, we simply replace $u$ with the homogeneous component of $u$ of degree $(i,v)$ which is in $U$ since $U$ is graded.  So we have $u= Y^\alpha e_i +u_1$ where $u_1$ is a $A_0$-linear combination of monomials of degree $(i,v)$ that are $<  Y^\alpha e_i$. Hence, by assumption, $u_1\in U_{(i,v)}$.  It follows that 
 $Y^\alpha e_i=u-u_1\in  U_{(i,v)}$, a contradiction. 
 \end{proof} 

The fact that $A$ has no elements of degree $(i,0)\in \ZZ^2$  with $i\neq 0$ has  an important consequence. 

\begin{lemma} 
\label{easy}
Let  $N$ be a  $\ZZ^2$-graded and finitely generated $A$-module.  Then   $\rho_N(v)$ is eventually either  a linear function of $v$ with leading coefficient in $\{d_1,\dots, d_g\}$ or  $-\infty$. 
\end{lemma} 

\begin{proof}  
First we observe that if $n$ is a generator of $N$ of degree,  say,  $(\alpha,\beta)\in \ZZ^2$,  then 
$Y_1^{\alpha_1}\cdots Y_g^{\alpha_g}n$ has degree $(\sum_j \alpha_j d_j+\alpha, \sum_j \alpha_j+\beta)$. Hence $N_{(i,v)}$ is non-zero only if $(i,v)=(\sum_j \alpha_j d_j+\alpha, \sum_j \alpha_j+\beta)$  for some $(\alpha_1,\dots, \alpha_g)\in \NN^g$ and some $(\alpha,\beta)$ degree of a minimal generator of $N$. If we set $D= \max\{d_1,\dots, d_g\}$, then $N_{(i,v)}\neq 0$ implies $\alpha\leq i\leq (v-\beta)D+\alpha$  for some degree $(\alpha,\beta)$  of a minimal generator of $N$.  As the module $N$  is finitely generated, it follows that  $\{ i\in \ZZ :  N_{(i,v)}\neq 0\}$ is finite for every $v\in \ZZ$.  
To prove that $\rho_N(v)$ is  either eventually linear in $v$ or $-\infty$, we present $N$ as $F/U$ where $F$ is a finitely generated  $A$-free bigraded module and   $U$ is a bigraded $A$-submodule of $F$.  Let $<$ be a monomial order  on $F$. Then $\rho_{F/U}(v)= \rho_{F/\ini_<(U)}(v)$. Hence  we may assume right away that $U$ is generated by monomials (with coefficients). 
We can consider a bigraded  chain of submodules 
$$0=N_0\subset N_1 \subset N_2  \subset \dots \subset N_p=N$$ with  cyclic  quotients
$C_i=N_i/N_{i-1}$  annihilated by a monomial prime ideal, i.e., an  ideal of the form $pA+J$ where $p$ is a prime ideal of $A_0$ and $J$ is an ideal generated by a subset of the variables $Y_1,\dots, Y_g$. 
It follows that 
$$
\rho_N(v)=\max\{ \rho_{C_i}(v) : i=1,\dots, p\}.
$$
Since the maximum of finitely many eventually linear  functions in one variable is an eventually linear function, it is enough to prove the statement for each $C_i$. That is, we may assume that, up to a shift $(-w_1,-w_2)\in \ZZ^2$, the module   $N$ has the form 
$A/P$ with $P=pA+J$ where $p$ is a prime ideal of $A_0$ and $J$ is generated by a subset of the variables.  With  $G=\{ i : Y_i\not\in P\}$, we have 
 $$\rho_N(v)=
\left\{\begin{array}{ll}
\max\{ d_i : i\in G\} (v-w_2)+w_1  & \mbox{  if  } G \neq \emptyset \mbox{ and  }   v\geq w_2, \\
-\infty   &\mbox{  if  } G=\emptyset \mbox{ and  }   v> w_2. \\  
\end{array}
\right.
$$
\end{proof}

\section{Regularity and powers} 
We return to the notation of Section \ref{CM}.  For a finitely generated graded $R$-module $M$ 
and a homogeneous ideal $I$ of $R$ we will study the behaviour of $\reg(I^vM)$ as a function of $v\in \NN$. 
For simplicity we will assume throughout that $I^vM\neq 0$ for every $v$. 
Let us consider the Rees algebra $\Rees(I)$ of $I$:

$$\Rees(I)=\bigoplus_{v\in \NN} I^v$$
with its natural bigraded structure given by 

$$\Rees(I)_{(i,v)}=(I^v)_i.$$
The Rees module of the pair $I,M$

$$\Rees(I,M)=\bigoplus_{v\in \NN} I^vM$$
 is clearly a finitely generated $\Rees(I)$-module naturally bigraded by
$$\Rees(I,M)_{(i,v)}=(I^vM)_i.$$
Let $f_1,\dots, f_g$  be a set of minimal homogeneous generators of $I$ of degrees, say, $d_1,\dots, d_g\in \NN$.
We may present  $\Rees(I)$ as a quotient of 
$$B=R[Y_1,\dots, Y_g]$$
via the map 
$$\psi:B\to \Rees(I), \quad Y_i\to f_i \in I_{d_i}=\Rees(I)_{(d_i,1)}.$$
Actually $B$ is naturally bigraded if we assign bidegree $(i,0)$  to $x\in R_i$ as an element of $B$ and 
by set $\deg Y_j=(d_j,1)$. 

Consider the extension $Q_RB$ of  $Q_R$ to $B$ and the Koszul homology 
$$H(Q_RB, \Rees(I,M))=H(Q_R, \Rees(I,M))=\bigoplus_{v\in \NN} H(Q_R,I^vM).$$
Since $Q_R H(Q_R, \Rees(I,M))=0$  $H(Q_R, \Rees(I,M))$ acquires naturally the structure of finitely generated 
$\ZZ^2$-graded $B/Q_RB$-module.  Here  
$$B/Q_RB=R_0[Y_1,\dots, Y_g]$$ 
has a  bigraded structure defined in Section \ref{nonstan}. 
Now for $i=0,\dots,n$ we let 
$$t_i(M)=\sup\{ j : H_i(Q_R, M)_j\neq 0\}.$$ 
 We have: 
 
 \begin{theorem}
 \label{main}
  Let $I$ be a  homogeneous ideal of $R$  minimally generated by  homogeneous elements of degree $d_1,\dots, d_g$ and $M$ be a finitely generated graded $R$-module. 
 Then there exist $\delta\in \{ d_1,\dots, d_g\}$ and  $c\in \ZZ$ such that   
  $$\reg(I^vM)=\delta v+c  \mbox{ for } v\gg 0.$$
  \end{theorem}
  
  \begin{proof} 
  For  $i=0,\dots,n$ consider the $i$-th Koszul homology module: 
  $$H_i=H_i(Q_R, \Rees(I,M))=\bigoplus_{v\in \NN} H_i(Q_R,I^vM).$$
  As already observed $H_i$ is a finitely generated $\ZZ^2$-graded  $B/Q_RB$-module. 
 Furthermore $\rho_{H_i}(v)=t_i(I^vM)$.  Therefore  we may apply \ref{easy} and have that either 
 $H_i(Q_R,I^vM)=0$ for large $v$ or $t_i(I^vM)$ is a linear function of $v$ for large $v$ with leading coefficient in  $\{ d_1,\dots, d_g\}$. As $\reg(I^vM)=\max\{ t_i(I^vM)-i : i=0,\dots, n\}$ we may conclude that $\reg(I^vM)$ is eventually a linear function in $v$ with leading coefficient in $\{ d_1,\dots, d_g\}$.
 \end{proof} 
 
 Theorem \ref{main} has been proved in \cite{CHT} and \cite{K} when $R$ is a polynomial ring over a field and in \cite{TW} for general base rings. Our proof is a modification (and a slight simplification) of the one given in  \cite{CHT}. Here and also  in Section \ref{BigCM}  our work was  largely inspired by  the papers of  Chardin on the subject, in particular by \cite{Ch1,Ch2,Ch3,Ch4}.  The $\delta$ appearing in  \ref{main} can be characterized in terms of minimal reductions, see \cite{K,TW} for details. The nature of the others invariants arising from \ref{main}, i.e., the constant term $c$ and the least $v_0$ such that the formula holds for each $v\geq v_0$,  have been  deeply  investigated in \cite{BCH,  Ch2, Ch4, EH, EU}  and is relatively well understood in small dimension but remains largely unknown in general.

 \section{Linear powers}
  Assume now that  the minimal generators of  $I$ have all degree $d$ and that  the minimal generators of  $M$ have all degree $d_0$. 
 Hence $I^vM$ is generated by elements of degree $vd+d_0$ and therefore  $\reg(I^vM)\geq vd+d_0$ for every $v$.   
 
 \begin{definition} 
 We say that $I$ has linear powers with respect to $M$ if  $\reg(I^vM)=vd+d_0$ for every $v$.   
\end{definition} 

When $R_0$ is a field, $I$ has linear powers with respect to $M$  if and only if for every $v$ the matrices representing the maps in the minimal $S$-free resolution of $I^vM$ have entries of degree $1$. 

We will give a characterization of linear powers in terms of the homological properties of the Rees module $\Rees(I,M)$. Note that, under the current assumptions, $\Rees(I)$ and $\Rees(I,M)$ can be given a compatible and ``normalized" $\ZZ^2$-graded structure: 
\begin{align*}
\Rees(I)_{(i,v)}&=(I^v)_{vd+i},\\
\Rees(I,M)_{(i,v)}&=(I^vM)_{vd+d_0+i}.
\end{align*}
From the presentation point of view, this amounts to set $\deg Y_i=(0,1)$ so that 
$B=R[Y_1,\dots, Y_g]$ is a $\ZZ^2$-graded $R_0$-algebra with generators in degree $(1,0)$, the elements of $R_1$,  and in degree $(0,1)$, the $Y_i$'s.  With the notations introduced in Section \ref{BigCM}, we have that $\Rees(I,M)^{(*,v)}=(I^vM)(vd+d_0)$. So, applying Proposition \ref{p:bigraded}: 
$$\reg_{(1,0)} \Rees(I,M)=\max\{ \reg \Rees(I,M)^{(*,v)} : v\in \NN\}=\max\{ \reg I^vM-vd-d_0 : v\in \NN \}.$$
Summing up we have:

\begin{theorem} 
\label{linearpow}  
\begin{itemize} 
\item[(1)] $\reg I^vM \leq vd+d_0+\reg_{(1,0)} \Rees(I,M)$ for all $v$ and the equality holds for at least one $v$. 
\item[(2)] $I$ has linear powers with respect to $M$ if and only if $\reg_{(1,0)} \Rees(I,M)=0$. 
\end{itemize}
\end{theorem} 

When $R$ is the polynomial ring over a field and $M=R$ Theorem \ref{linearpow} part (2)   has been proved in \cite{BCV} extending earlier results of R\"omer \cite{R}. 

Theorem \ref{linearpow} and Theorem \ref{main}  have been  generalized to the case where the single ideal $I$ is replaced with a set of ideals $I_1,\dots, I_p$ and one looks at the regularity $\reg(I_1^{v_1}\cdots I_p^{v_p}M)$ as a function of $(v_1,\dots, v_p)\in \NN^p$. The main difference is that $\reg(I_1^{v_1}\cdots I_p^{v_p}M)$ is (only) a piecewise linear function unless each ideal $I_i$ is   generated in a single degree, see \cite{BC1, BC2, G} for details.

 \end{document}